\documentclass[a4paper,UKenglish]{amsart}
\relax
\pdfoutput=1

\usepackage[UKenglish]{babel}

\usepackage[utf8]{inputenc}
\usepackage{amsmath, amsxtra, amsfonts, amssymb}%  MnSymbol
\usepackage{hyperref}
\usepackage{url}
\usepackage{mathtools} %%for coloneqq
\usepackage{amsthm}
\usepackage{csquotes}
\usepackage[backend=biber,style=numeric,maxbibnames=99]{biblatex}
\newcommand{\fR}{\mathfrak{R}}

\newcommand{\fM}{\mathfrak{M}}

\newcommand{\fC}{\mathfrak{C}}
\newcommand{\fB}{\mathfrak{B}}
\newcommand{\fA}{\mathfrak{A}}

\newcommand{\Q}{\mathbb{Q}}
\newcommand{\N}{\mathbb{N}}
\newcommand{\Z}{\mathbb{Z}}

\DeclareMathOperator{\Th}{Th}

\DeclareMathOperator{\Aut}{Aut}

\DeclareMathOperator{\Csp}{CSP}
\DeclareMathOperator{\CSP}{CSP}

\renewcommand{\succ}{\mathrm{succ}}

\DeclareMathOperator{\AND}{\wedge}

\newcommand{\ceq}{\coloneqq}

\newcommand{\cequiv}{\mathrel{\vcentcolon\equiv}}
\DeclarePairedDelimiterX\set[1]\lbrace\rbrace{#1} %has better spacing properties and can use *

\makeatletter
\def\moverlay{\mathpalette\mov@rlay}
\def\mov@rlay#1#2{\leavevmode\vtop{%
   \baselineskip\z@skip \lineskiplimit-\maxdimen
   \ialign{\hfil$\m@th#1##$\hfil\cr#2\crcr}}}
\newcommand{\charfusion}[3][\mathord]{
    #1{\ifx#1\mathop\vphantom{#2}\fi
        \mathpalette\mov@rlay{#2\cr#3}
      }
    \ifx#1\mathop\expandafter\displaylimits\fi
    }
\makeatother

\newcommand{\cupdot}{\charfusion[\mathbin]{\cup}{\cdot}}
\newcommand{\bigcupdot}{\charfusion[\mathop]{\bigcup}{\cdot}}

\theoremstyle{definition}
\newtheorem{theorem}{Theorem}[section]
\newtheorem{proposition}[theorem]{Proposition}
\newtheorem{definition}[theorem]{Definition}
\newtheorem{lemma}[theorem]{Lemma}
\newtheorem{corollary}[theorem]{Corollary}
\newtheorem{remark}[theorem]{Remark}
\newtheorem{observation}[theorem]{Observation}
\newtheorem{example}[theorem]{Example}

\hypersetup{pdfinfo={
Title={Tractable Combinations of Theories via Sampling},
Author={Manuel Bodirsky, Johannes Greiner}
}}

\title{Tractable Combinations of Theories via Sampling}

\thanks{Both authors have received funding from the European Research Council (ERC  Grant Agreement no. 681988, CSP-Infinity), and the DFG Graduiertenkolleg 1763 (QuantLA)}

\author{Manuel Bodirsky, Johannes Greiner}
\address{Institut f\"ur Algebra\\
    Technische Universit\"at Dresden}
\email{manuel.bodirsky@tu-dresden.de, johannes.greiner@tu-dresden.de}

\begin{document}
\begin{abstract}
For a first-order theory $T$, the Constraint Satisfaction Problem of $T$ is the computational problem of deciding whether a given conjunction of atomic formulas is satisfiable in some model of $T$. In this article we develop sufficient conditions for polynomial-time tractability of the constraint satisfaction problem for the union of two theories with disjoint relational signatures. To this end, we introduce the concept of sampling for theories and show that samplings can be applied to examples which are not covered by the seminal result of Nelson and Oppen.
\end{abstract}

\maketitle

\section{Introduction}\label{sec:introduction}

Reasoning tasks for intelligent agents often require to check whether certain configurations or situations are legitimate or possible. Such decision problems can often be modelled as \emph{Constraint Satisfaction Problems (CSPs)}. The CSP of a first-order theory $T$ with finite relational signature $\tau$ is the computational problem of deciding whether a set of atomic $\tau$-formulas is satisfiable in some model of $T$. In fact, any decision problem is polynomial-time Turing equivalent to a CSP~ \cite{BodirskyGrohe}.
Often, reasoners have to decide instances of $\CSP(T_1 \cup T_2)$ where $T_1$ and $T_2$ are first-order theories with disjoint relational signatures such that $\CSP(T_1)$ and $\CSP(T_2)$ are well understood. This problem has already be studied by Nelson and Oppen~\cite{NelsonOppen79} and many others have continued this line of research (see for example Baader and Schulz~\cite{BaaderSchulz01}). CSPs of unions of theories are at the heart of SMT-Solvers (SAT Modulo Theories) and occur frequently in software verification~\cite{MouraBjorner11_SMTIntroAndApplications}.

The results of Nelson and Oppen~\cite{NelsonOppen79, Oppen80_ComplexityConvexityAndCombinationsOfTheories} provide sufficient conditions for the polynomial-time tractability of $\CSP(T_1 \cup T_2)$, covering a great variety of theories. 
Schulz~\cite{Schulz00_WhyCombinedDecisionProblemsAreOftenIntractable} as well as Bodirsky and Greiner~\cite{BodirskyGreinerLMCS19} have shown that in many situations, the conditions of Nelson and Oppen are also necessary for polynomial-time tractability (unless P = NP).

In this article we will present new sufficient conditions for polynomial-time tractability of $\CSP(T_1 \cup T_2)$, and present examples which show that our conditions are incomparable to the conditions provided by Nelson and Oppen.
 To prove our results, we introduce the notion of \emph{sampling} for theories and show how to apply samplings in order to solve $\CSP(T_1 \cup T_2)$ in polynomial time. 
 In particular, we construct a sampling for $T_1 \cup T_2$ from given samplings for $T_1$ and $T_2$.

\section{Sampling for a Theory}\label{sec:samplingForTheories}

Let $\tau$ be a finite \emph{relational signature}, i.e., a finite set of relation symbols, each of which endowed with a natural number denoting its arity. A \emph{$\tau$-structure} $\fA$ is a tuple starting with a set $A$, followed by an $n$-ary relation $R^{\fA} \subseteq A^n$ for each symbol $R\in \tau$ of arity $n$. We will always denote structures with uppercase fraktur letters and their domain with the corresponding uppercase Latin letter. 
A \emph{theory} is a set of first-order sentences. A \emph{$\tau$-theory} is a theory
where all non-logical symbols are in $\tau$. By $\Th(\fA)$ we denote the \emph{theory of $\fA$}, i.e., the set of all first-order sentences which are satisfied in $\fA$. A $\tau$-structure $\fA$ is a \emph{model} of a $\tau$-theory $T$, if all sentences of $T$ are satisfied in $\fA$. An \emph{atomic $\tau$-formula} is either of the form $R(x_1, \dots, x_n)$ with $R\in \tau$ of arity $n$ and variables $x_1, \dots, x_n$, of the form $x_1=x_2$, or of the form $\bot$, indicating the logical \enquote*{false}. 

We now present the central definitions of this article.
\begin{definition}
Let $\tau$ be a finite relational signature and $T$ a $\tau$-theory.
\begin{itemize}
 \item A \emph{sampling} for $T$ is a sequence $(L_n)_{n\in \N}$ of finite sets of finite $\tau$-structures such that for all $n\in \N$ and all conjunctions of atomic $\tau$-formulas $\phi$ with at most $n$ variables, there exists a model for $T$ in which $\phi$ is satisfiable if and only if there exists $\fB\in L_n$ such that $\phi$ is satisfiable in $\fB$.
 An element of $L_n$ is called \emph{$n$-sample}, or simply \emph{sample}, for $T$. 

 \item Let $(L_n)_{n\in \N}$ be a sampling for a theory $T$. 
 %We call $\lvert L_n \rvert \ceq \sum_{\fB\in L_n} \lvert B \rvert$ the \emph{size} of $L_n$.
 We define $\lvert L_n \rvert \ceq \sum_{\fB\in L_n} \lvert B \rvert$.
 A sampling $(L_n)$ is \emph{polynomial} if there exists a polynomial $p \in \Z[x]$ such that $\lvert L_n \rvert \leq p(n)$ for all $n\in \N$.
 \item A sampling $(L_n)_{n\in \N}$ is \emph{computable} if there exists an algorithm that, given $n$, computes $L_n$. If this algorithm runs in time polynomial in $n$, then $(L_n)$ is \emph{computable in polynomial time}.
 \item A \emph{sampling} for a structure $\fA$ is a sampling for $\Th(\fA)$.
 \item  For $n\geq 1$, let $S_n(T)$ be the class of all finite $\tau$-structures $\fB_n$ such that a conjunction of atomic $\tau$-formulas with at most $n$ variables is satisfiable in $\fB_n$ if and only if it is satisfiable in some model of $T$.
\end{itemize}
\end{definition}

While we are not aware that samplings have been defined for theories before, there is a definition for structures by Bodirsky, Macpherson and Thapper~\cite{BodirskyMacphersonThapper13} and we will now compare the two notions.
We need the classical notion of a CSP which is defined for structures instead of theories. Let $\tau$ be a finite relational signature and $\fA$ a
$\tau$-structure. $\CSP(\fA)$ is the computational problem of deciding
whether a conjunction of atomic $\tau$-formulas is satisfiable in $\fA$.
Notice that $\CSP(\Th(\fA)) = \CSP(\fA)$ holds for all relational
structures $\fA$, because $\Th(\fA)$ explicitly states which instances
are satisfiable.
Therefore, the definition of \enquote{sampling for $\fA$} in~\cite{BodirskyMacphersonThapper13} can be attained from our definition by imposing the  additional conditions on $(L_n)$ that for all $n\in \N$ the set $L_n$ may only contain one structure $\fB_n$ and that $\fB_n$ needs to have a homomorphism to $\fA$.
 Hence, any structure which has \enquote{efficient sampling} in the definition of Bodirsky, Macpherson and Thapper has a sample which is computable in polynomial time in our definition, and is thus polynomial.
 
 We will now prove that samplings for \emph{structures} only require one element in each $L_n$, just like in the definition of Bodirsky, Macpherson and Thapper. However, we will see in Example~\ref{ex:Succ2Col} that dropping the requirement that $\fB_n$ is homomorphic to $\fA$ allows smaller samples in some cases and can therefore reduce the runtime of algorithms that run on the samples.

\begin{definition}
 Let $T$ be a theory with finite relational signature. $T$ has the \emph{Joint Homomorphism Property (JHP)} if for any two models $\fA$, $\fB$ of $T$ there exists a model $\fC$ of $T$ such that both $\fA$ and $\fB$ homomorphically map to $\fC$. 
\end{definition}

\begin{proposition}[Proposition~2.1 in~\cite{BodirskyGreinerLMCS19}]\label{thm:universalModelJHP}
Let $T$ be a theory with finite relational signature. Then there exists a model $\fA$ of $T$ such that $\CSP(\fA) = \CSP(T)$ if and only if 
$T$ has the \emph{Joint Homomorphism Property (JHP)}.
\end{proposition}

\begin{definition}
For a set $L$ of $\tau$-structures  we define the \emph{disjoint union} $\bigcupdot_{\fB \in L} \fB$ as the structure with domain $\bigcupdot_{\fB \in L} B$ and a $k$-ary relation $R\in \tau$ holds on a tuple $t$ if and only if there exists $\fB \in L$ such that $t \in B^k$ and $R(t)$ holds in $\fB$.
\end{definition}

\begin{proposition}\label{thm:samplesOfStructures}
Let $T$ be a theory with JHP and let $(L_n)_{n\in \N}$ be a sampling for $T$. Then $(L'_n)_{n\in \N}$ with  $L'_n \ceq \set{\bigcupdot_{\fB\in L_n} \fB}$ is also a sampling for $T$. Moreover, $\lvert L'_n \rvert = \lvert L_n \rvert$ for all $n\in \N$.
\end{proposition}
\begin{proof}
Let $\phi$ be a conjunction of atomic $\tau$-formulas with variables in $V \ceq \set{x_1, \dots, x_n}$.
If $\phi$ is satisfiable in $T$, then it is satisfiable in some $\fB \in L_n$ and therefore in $\fC \ceq \bigcupdot_{\fB\in L_n} \fB$. 
Assume that $\phi$ is satisfiable in $\fC$. Let $s\colon V \rightarrow C$ be a satisfying assignment. Then $s$ induces a partition $V_1, \dots, V_k$ on $V$, where two elements of $V$ are in the same part if they are mapped into the same element of $L_n$. However, $\phi$ cannot contain a conjunct $R(\dots, x_i, \dots, x_j, \dots)$ with $x_i \in V_i$ and $x_j \in V_j$ for $i\neq j$, as this relation would not hold in $\fC$. Hence, $\phi$ is equivalent to a conjunction $\phi_1 \AND \dots \AND \phi_k$, where each $\phi_i$ is itself a conjunction of atomic $\tau$-formulas, $s$ maps all variables of $\phi_i$ into $V_i$ and $\phi_i$ and $\phi_j$ have no common variables for $i\neq j$. 
Let $\fA$ be a model of $T$ such that $\CSP(\fA) = \CSP(T)$, which exists by Proposition~\ref{thm:universalModelJHP}.
As $(L_n)_{n\in \N}$ is a sampling for $T$, it is also a sampling for $\fA$ and therefore, each $\phi_i$ is satisfiable in $\fA$. Hence, $\phi$ is satisfiable in $\fA$ and therefore in $T$.
We have $\lvert L'_n \rvert  = \sum_{\fB \in L_n} \lvert B \rvert = \lvert L_n \rvert$.
\end{proof}

Due to Proposition~\ref{thm:samplesOfStructures}, and the fact that $\bigcupdot_{\fB\in L_n} \fB$ has the same size as $\vert L_n \rvert $, we will now often assume that samplings for structures contain only one element $\fB_n \in S_n(\Th(\fA))$ in $L_n$ (for all $n \in \N$). We will write $(\fB_n)_{n\in \N}$ instead of $(\set{\fB_n})_{n\in \N}$. It will become clear in Section~\ref{sec:samplingForCombinations} why we will not assume that sampling of structures make this assumption in general.

To justify why our definition of \enquote{sampling} allows $L_n$ to contain multiple structures in general, we would like to give an example of a theory $T$  which has a sampling but no sampling $(L_n)$ of $T$ has only one element in $L_n$ for some $n\geq 2$.

\begin{example}[similar to Example~2.2 in~\cite{BodirskyGreinerLMCS19}]\label{ex:noJHP}
Let $\tau$ be a signature consisting of the unary symbols $O,P,Q$ and the binary symbol $I$. Let $T$ be the set of the following sentences
\begin{align*}
 &\forall x,y.((O(x) \AND O(y))\Rightarrow x=y),\\
 &\forall x.\neg (P(x) \AND Q(x)), \quad \text{and}\\
 &\forall x,y.(I(x,y) \Leftrightarrow \neg (x=y)).
\end{align*}
The following is a sampling for $T$ which is computable in polynomial time: For $n\in \N$ let $L_n$ consist of $\fB_1$ and $\fB_2$ defined on $\set{1,\dots, 2n}$ where $P^{\fB_1} = P^{\fB_2} \ceq \set{1,\dots, n}$, $Q^{\fB_1} = Q^{\fB_2} \ceq \set{n+1, \dots, 2n}$, $O^{\fB_1} \ceq\set{1}$ and $O^{\fB_2} \ceq \set{n+1}$. In both structures, the $I$-relation denotes the inequality relation.
While it is easy to check that $(L_n)_{n\in \N}$ is a sampling for $T$, there is no sampling with only one element in $L_n$ for some $n\geq 2$. To prove this, consider the formulas $\phi_1 \cequiv O(x) \AND P(x)$, which is satisfiable in $\fB_1$, and $\phi_2 \ceq O(y) \AND Q(y)$, which is satisfiable in $\fB_2$. Assume that both are satisfiable in some $\tau$-structure $\fB$. Then there exists $a \in O^{\fB}\cap P^{\fB}$ and $b\in O^{\fB}\cap Q^{\fB}$. 
If $a=b$, then $\psi_1 \ceq P(x) \AND Q(x)$ is satisfiable in $\fB$. 
If $a \neq b$, then $\psi_2 \ceq O(x)\AND O(y) \AND I(x,y)$ is satisfiable in $\fB$ or $\psi_3 \ceq I(x,y) \AND x=y$ is satisfiable in $\fB$ (in case $I^{\fB}$ is not the complement of equality). However, none of $\psi_1, \psi_2, \psi_3$ is satisfiable in some model of $T$, contradiction.
\end{example}

To characterise theories with a computable sampling, we need the following definition. 
\begin{definition}
Let $\tau$ be a relational signature and let $\phi$ be a conjunction of atomic  $\tau$-formulas with variables $x_1, \dots, x_n$. The \emph{canonical database} $D(\phi)$ is defined as the structure with domain $x_1, \dots, x_n$ such that a tuple $t$ is in a relation $R^{D(\phi)}$ for $R \in \tau$ if and only if $R(t)$ is a conjunct of $\phi$.
\end{definition}

Clearly, there is a homomorphism from $D(\phi)$ to $\fA$ if and only if $\phi$ is satisfiable in $\fA$.

\begin{proposition}\label{thm:decidableEqSampling}
 Let $T$ be a theory with finite relational signature. Then, $T$ has a computable sampling if and only if $\CSP(T)$ is decidable.
\end{proposition}
\begin{proof}
Let $\tau$ be the signature of $T$.
 Suppose there exists a computable sampling $(L_n)_{n\in \N}$ for $T$. Let $\phi$ be a conjunction of atomic $\tau$-formulas with at most $n$ variables. Then there are only finitely many maps from $D(\phi)$ to structures in $L_n$ and for each of them we can determine whether it is a homomorphism or not. Therefore $\CSP(T)$ is decidable. 
 
 Conversely, suppose that $\CSP(T)$ is decidable. There are only finitely many atomic formulas on $n$ variables because $\tau$ is finite and the arity of each relation is finite. Hence, there are only finitely many conjunctions of atomic formulas with at most $n$ variables such that no conjunct is repeated. Using the decision procedure for $\CSP(T)$, we may determine which of the conjunctions are satisfiable in some model of $T$ and define $L_n$ as the set of all canonical databases of the satisfiable conjunctions. Then, by construction, $(L_n)$ is a sampling for $T$.
\end{proof}

To determine whether a theory has a polynomial sampling the following general observation is helpful. 
\begin{observation}\label{thm:HomToStruct_Substruct}
Let $\fA$ be a  structure with finite relational signature $\tau$.
By Proposition~\ref{thm:samplesOfStructures}, determining whether a polynomial sampling for $\fA$  exists is equivalent to determining whether for all $n\in \N$ there exists $\fB_n \in S_n(\Th(\fA))$  such that $(\fB_n)$ is a polynomial sampling. If $\fB_n \in S_n(\Th(\fA))$ and there exists a homomorphism $h \colon \fB_n \rightarrow \fA$, then $h(\fB_n) \in S_n(\Th(\fA))$ and $\lvert h(B_n) \rvert \leq \lvert B_n \rvert$.
  If we can prove that each element of $S_n$ homomorphically maps to $\fA$ we can therefore assume that $L_n$ consists of a single substructure of $\fA$. An element of $S_n(\Th(\fA))$ which is a substructure of $\fA$ is called an \emph{$n$-universal substructure} of $\fA$. Minimal sizes of $n$-universal substructures are easier to determine and have been explored in the past \cite{Moon64_OnMinimalNUniversalGraphs, GoldbergLivshits68_OnMinimalUniversalTrees, LozinRudolf07_MinimalUniversalBipartiteGraphs}.

%However, the image of $h$ might not be computable in a generic way. 
\end{observation}
In general, $S_n(\Th(\fA))$ may contain structures which are smaller than the smallest $n$-universal substructure of $\fA$ and Example~\ref{ex:Succ2Col} will provide an example for this case.

To exemplify the use of Observation~\ref{thm:HomToStruct_Substruct} we will now use it to prove that a certain theory does not have a polynomial sampling. The theory $T$ in Lemma~\ref{thm:radoNoSampling} is contained in the theory of every undirected, loopless graphs with a binary relation symbol $E$ for the edge relation
and a binary relation symbol $N$ for the set of all pairs of distinct elements that are not related by an edge.

\begin{lemma}\label{thm:radoNoSampling}
 Let $\tau$ consist of the binary relation symbols $E$ and $N$ and let $T$ contain only the sentence 
 \begin{align*}
  \forall x,y.\big(&\neg E(x,x) \AND \neg(E(x,y) \AND \neg E(y,x)) \\
  &\neg N(x,x) \AND \neg (E(x,y) \AND N(x,y)) \big).
 \end{align*}
 Then there is a sampling for $T$ but not a polynomial sampling.
\end{lemma}

\begin{proof} 
  First notice that $\CSP(T)$ is decidable (in polynomial-time) because it suffices to check, after eliminating equalities by substitution of variables, that the instance does not contain $\bot$, $E(x,x)$, $N(x,x)$, $E(x,y) \AND N(x,y)$ or $E(y,x) \AND N(x,y)$ for any two variables $x,y$.
  Therefore, there exists a sampling $(L_n)$ for $T$ by Proposition~\ref{thm:decidableEqSampling}.
  Fix $n \geq 2$. 
  Notice that the disjoint union of two models of $T$ is again a model of $T$. Therefore, $T$ has the JHP and we may without loss of generality assume that $L_n$ contains only one structure $\fB$ due to Proposition~\ref{thm:samplesOfStructures}.
  For the structure $(B; E^{\fB})$ we know that it is without loops as otherwise $E(x,x)$ is satisfiable in $\fB$, contradiction. 
  Furthermore, we may assume that $(B; E^{\fB})$ is an undirected graph, because adding the reverse  direction to an edge does not increase the size of $\fB$ and does not change the satisfiability of instances in $\fB$. To prove the latter, notice that instances satisfiable in $\fB$ cannot become unsatisfiable in $\fB'$, where $E(a,b)$ has been added, because all homomorphisms to $\fB$ are also homomorphisms to $\fB'$. Suppose there exists an instance $\phi$ which is not satisfiable in $\fB$ but there exists a homomomorphism $h\colon D(\phi) \rightarrow \fB'$. Let $\phi'$ be the instance where all conjuncts $E(x,y)$ in $\phi$ with $h(x) = a$ and $h(y)=b$ are replaced by $E(y,x)$. Then $h$ is a homomomorphism from $D(\phi')$ to $\fB$, and therefore $\phi'$ is satisfiable in some model of $T$. As $E$ is symmetric in all models of $T$, $\phi$ is satisfiable in some model of $T$ and therefore in $\fB$, contradiction.
    
  Lastly, we know that $N^{\fB} \subseteq B^2 \setminus E^{\fB}$ must hold, as $E(x,y) \AND N(x,y)$ is satisfiable in $\fB$ otherwise.   
  Therefore, $\fB$ is homomorphic to a model $\fR$ of $T$ in which $\forall x,y.(N(x,y) \Leftrightarrow x\neq y \AND \neg E(x,y))$ holds and which contains all undirected loopless finite graphs as induced subgraphs. Then $\lvert \fB \rvert$ is at least the size of the smallest $n$-universal substructure of $\fR$ by Observation~\ref{thm:HomToStruct_Substruct}. However, due to results by Noga Alon~\cite{Alon17_asymtoticallyoptimalinduceduniversalgraphs} we know that any substructure of $\fR$ which contains all graphs of size at most $n$ as induced subgraphs has size exponential in $n$ and that this does not depend on the choice of $\fR$.
\end{proof}

Theories with polynomial samplings include all theories of finite structures, the theory of successor on the natural numbers, and the theories of structures in Lemma~\ref{thm:QhasSampling} and Lemma~\ref{thm:baseStrHaveSampling} below.

\section{Sampling for Unions of Theories}\label{sec:samplingForCombinations}

Let $T_1, T_2$ be two theories with disjoint relational signatures and that $T_1$ and $T_2$ have polynomial samplings. In this section we will present sufficient conditions on $T_1$ and $T_2$ such that $T_1 \cup T_2$ has polynomial sampling.
To this end, we will need to construct models for $T_1 \cup T_2$.  Let $\tau_1, \tau$ be relational signatures with $\tau_1 \subseteq \tau$ and $\fA$ a $\tau$-structure. The \emph{$\tau_1$-reduct} of $\fA$, written $\fA^{\tau_1}$, is the structure we get when all relations with symbols not contained in $\tau_1$ are removed from $\fA$. 
The following definition goes back to Tinelly and Ringeisen~\cite{TinelliRingeissen03_UnionsOfNonDisjointTheoriesAndCombinationsOfSatisfiabilityProcedures}. 
\begin{definition}
 Let $\fA_1$ and $\fA_2$ be $\tau_1$ and $\tau_2$ structures respectively. A $\tau_1\cup\tau_2$-structure $\fA$ is a \emph{fusion} of $\fA_1$ and $\fA_2$ if and only if $\fA^{\tau_i}$ is isomorphic to $\fA_i$ for $i=1$ and $i=2$.
\end{definition}

\begin{proposition}[Proposition~1 and Lemma~1 in~\cite{BaaderSchulz01}]\label{thm:fusionIsModel}
 For $i=1,2$, let $T_i$ be a $\tau_i$ theory. 
 A $\tau_1 \cup \tau_2$-structure $\fA$ is a model of $T_1 \cup T_2$ if and only if $\fA$ is a fusion of a model for $T_1$ and a model for $T_2$. Furthermore, two structures have a fusion if and only if their domains have the same cardinality.
\end{proposition}

The proof of Lemma~1 in~\cite{BaaderSchulz01} essentially argues that any bijection between the domains of two structures with disjoint signatures defines a fusion of the two structures, which is a fact we will use later on.

It is an easy observation that an instance $\phi_1 \AND \phi_2$ of $\CSP(T_1 \cup T_2)$ is satisfiable if and only if for $i=1$ and $i=2$ there exist models $\fA_i$ of $T_i$ with $\lvert A_1 \rvert = \lvert A_2 \rvert$ such that $\phi_i$ is satisfiable in $\fA_i$ and the satisfying assignments of $\phi_1$ and $\phi_2$ identify exactly the same variables. 
The following notion describes a property of a sampling that allows us to transfer identifications of variables between models and samples.

\begin{definition}
 We call a sampling $(L_n)_{n\in \N}$ for $T$ \emph{equality-matching} if for all $n\in \N$ and all conjunctions of atomic $\tau$-formulas $\phi$ with variables in $\set{x_1, \dots, x_n}$ 
and for any conjunction $\psi$ of equalities and negated equalities on $\set{x_1, \dots, x_n}$, there exists a model of $T$ in which $\phi \AND \psi$ is satisfiable if and only if there exists $\fB\in L_n$ such that $\phi \AND \psi$ is satisfiable in $\fB$.
 A sampling is \emph{equality-matching} for a structure $\fA$ if it is equality-matching for $\Th(\fA)$.
\end{definition}

To be equality-matching will be the property that samplings for $T_1$ and $T_2$ must satisfy in order to construct a sampling for $T_1 \cup T_2$ in Theorem~\ref{thm:Both_NoAlg+InjSamp}. 
An example for an equality-matching sampling is the one constructed in Example~\ref{ex:noJHP} because $I$ is the negation of equality in all models of $T$ and the constructed samples and therefore, the instances of $\CSP(T)$ themselves can specify how to identify variables.  Also, the sampling constructed in the proof of Proposition~\ref{thm:decidableEqSampling} is equality-matching if there exists $I\in \tau$ such that in all models of $T$ we have $I(x,y) \Leftrightarrow \neg(x= y)$.

Another class of theories where an equality-matching sampling always exists are $\omega$-categorical theories. A theory is \emph{$\omega$-categorical} if it has only one countable model (up to isomorphism). 
For a structure $\fA$, the \emph{orbit} of $t \in A^k$ is $\set{\alpha(t) \mid \alpha\in \Aut(\fA)}$ where $\Aut(\fA)$ is the set of automorphisms of $\fA$ and $\alpha$ is applied componentwise.

\begin{proposition}
 Let $T$ be an $\omega$-categorical theory with finite relational signature. Then there exists an equality-matching sampling for $T$.
\end{proposition}
\begin{proof}
  If $T$ has an infinite model $\fA$ then there exists a countably infinite model $\fA'$ of $T$ such that $\Th(\fA) = \Th(\fA')$ by the Downward L\"owenheim-Skolem Theorem (see~\cite{HodgesLong}, page~90).
 If the unique countable model $\fA$ of $T$ is finite, then $(L_n)$ with $L_n \ceq \set{\fA}$ for all $n \in \N$ is an equality-matching sampling for $T$ and we are done.
 We may therefore assume that an instance $\phi$ in conjunction with equalities and disequalities is satisfiable in some model of $T$ if and only if it is satisfiable in a fixed countably infinite structure $\fA$.
By the theorem of Engeler, Ryll-Nardzewski and Svenonius (see~\cite{HodgesLong} page~341) there are only finitely many orbits of $n$-tuples in $\fA$. 
To construct an $n$-sample for $\fA$ we pick a tuple from each of the finitely many orbits of length $n$ and define $\fB_n$ to be the substructure of $\fA$ whose domain is the union of all elements appearing in these tuples. We claim that $(\fB_n)$ is an equality-matching sampling for $T$.

Let $\phi$ be a conjunction of atomic formulas and $\psi$ a conjunction of equalities and negated equalities, both on variables $x_1, \dots, x_n$. Clearly, if there exists a homomorphism from $D(\phi\AND \psi)$ to $\fB_n$ then there exists a homomorphism from $D(\phi \AND \psi)$ to $\fA$ because $\fB_n$ embedds into $\fA$. If $h\colon D(\phi\AND \psi) \rightarrow \fA$ is a homomorphism, then the image of $h$ is a tuple $t$ of length at most $n$. Then, there is a tuple $t'$ in $\fB_n$ with the same orbit as $t$. Hence, we can choose $\alpha \in \Aut(\fA)$ such that $\alpha(t) = t'$ and $\alpha\circ h$ is a homomorphism from $D(\phi\AND \psi)$ to $\fB_n$.
\end{proof}

It follows from Proposition~\ref{thm:decidableEqSampling}
that there are are $\omega$-categorical structures without computable sampling, because some of those structures have an undecidable CSP~\cite{BodirskyNesetrilJLC}.
There are also theories with a sampling computable in polynomial time, but without polynomial equality-matching sampling. 
Indeed, let $\tau$ contain the binary relations $E$ and $N$ and let $T$ contain only
 \begin{align*}
  \forall x,y \, \big(&\neg(E(x,y) \AND \neg E(y,x)) \AND \\
  &\neg (E(x,y) \AND N(x,y) \AND x\neq y)\big).
 \end{align*}
Then $L_n \ceq (\set{0}, E,N)$ with $E = N \ceq \set{(0,0)}$ for all $n\in \N$ yields a sampling for $T$. 
If there exists a polynomial equality-matching sampling for $T$, then there exists a polynomial sampling for $T'\ceq T \cup \Th(\Z; \neq)$ by Proposition~\ref{thm:InjSamp+ECL}. However, in $\CSP(T')$ we can define relations $E', N'$ which must satisfy the theory from Lemma~\ref{thm:radoNoSampling} by $E'(x,y) \ceq E(x,y) \AND x\neq y$ and $N'(x,y) \ceq N(x,y) \AND x\neq y$. Therefore, a polynomial equality-matching sampling for $T'$ cannot exist.

It is in general not true that if $(L_n)_{n\in \N}$ is equality-matching then $(L'_n)_{n\in \N} \ceq (\bigcupdot_{\fB\in L_n} \fB)_{n\in \N}$ is equality-matching as well. An example is $\fA = (\set{0,1}; < )$ with $L_4$ consisting of two copies of $\fA$. Then $x_1 <x_2 \AND x_3 < x_4 \AND \bigwedge_{i<j} x_i \neq x_j$ is not satisfiable in $\fA$ or $L_4$, but is satisfiable in $(L_4')$.
Even homomorphic images of equality-matching samplings, such as in Observation~\ref{thm:HomToStruct_Substruct}, need not be equality-matching.

For later use, we now present two classes of structures with equality-matching sampling computable in polynomial time.
If $\fB$ is a reduct of $\fA$ and all relations in $\fA$ have a first-order definition in $\fB$, then $\fA$ is called a \emph{first-order expansion} of $\fB$.

\begin{lemma}\label{thm:QhasSampling}
All reducts of  first-order expansions of  ${(\Q; <)}$ have an equality-matching sampling computable in polynomial time.
\end{lemma}
\begin{proof}
 For all $n\in \N$ let $\fB_n$ be the substructure of $(\Q;<)$ with domain $\set{1,2, \dotsc, n}$. Then $(\fB_n)$ is an equality-matching sampling for $(\Q; <)$. This is because the orbit of a tuple $t$ is fully determined by the order and identification of the coordinates of $t$ and all possible orderings and identifications of $n$ variables homomorphically map to $\fB_n$. Therefore, the constructed sampling is of the same kind as the one constructed in the proof of Proposition~3.4 and therefore equality-matching.
 
  Any relation which is first-order definable over a structure is a union of orbits of that structure. As the constructed sampling consist of substructures of $(\Q;<)$ and each $n$-sample includes all orbits of length $n$, the constructed sampling is also a sampling for reducts of first-order expansions of $(\Q; <)$.
\end{proof}

\begin{lemma}\label{thm:baseStrHaveSampling}
Let $P_1, \dots, P_m$ be a partition of $\Q$ where all parts are infinite and co-infinite. 
All reducts of  first-order expansions of
$(\Q; P_1, \dotsc, P_m)$ have an equality-matching sampling computable in polynomial time.
\end{lemma}

\begin{proof}
 For all $n\in \N$ define $\fB_n$ as the substructure of $(\Q; P_1, \dotsc, P_m)$ with domain $\set{a_{1,1}, \dotsc, a_{n,m}}$ where $a_{i,j} \in P_j$ for $1\leq i \leq n$ and $1 \leq j \leq m$. Then $(\fB_n)$ is an equality-matching sampling for $(\Q; P_1, \dotsc, P_m)$. This is because the orbit of an $n$-tuple $t$ in $(\Q; P_1, \dotsc, P_m)$ is fully determined by identifications of coordinates of $t$ and which coordinate belongs to which $P_i$. It is easy to see that all possible identifications and assignments of $n$ variables to $P_1, \dots P_m$  can be realized in $\fB_n$. Therefore, the constructed sampling is of the same kind as the one constructed in the proof of Proposition~3.4 and therefore equality-matching.
 
 The reason why $(\fB_n)$ is also an equality-matching sampling for reducts of first-order expansions of $(\Q; P_1, \dotsc, P_m)$ is the same as in the proof of Lemma~3.5.
\end{proof}

Another condition we need in order to construct a sampling for the union of two theories is a primitive positive version of the classical concept of \emph{no algebraicity} in model theory (see~\cite{HodgesLong}).
A formula is called \emph{primitive positive (pp)} if it is a conjunction of atomic formulas where some variables may be existentially quantified. 

\begin{definition}
A $\tau$-structure $\fA$  has \emph{no pp-algebraicity} if for any primitive positive $\tau$-formula $\phi(x)$ with parameters $a_1, \dots, a_n$ the set $\set{b \in A \mid \fA \models \phi(b)}$ is either contained in $\set{a_1, \dots, a_n}$ or infinite.
\end{definition}

It is easy to check that first-order expansions of $(\Q; <)$ or $(\Q; P_1, \dotsc, P_m)$ do not have pp-algebraicity because $a<x<b$ has either no or infinitely many satisfying assignments in $(\Q;<)$ (for all $a,b\in \Q$), and in $(\Q; P_1, \dots, P_m)$ any two distinct elements in $P_i$ are indistinguishable for first-order definable relations. When it is clear from context that we mean a set, we will use $[n]$ as a shorthand for $\set{1,2, \dots, n}$. 

We are now ready to present the first main result.

\begin{theorem}\label{thm:Both_NoAlg+InjSamp}
Let $T_1$ and $T_2$ be theories with finite relational and disjoint signatures $\tau_1, \tau_2$ respectively.
If all models of $T_1$ and $T_2$ do not have pp-algebraicity and $T_1$ and $T_2$ have equality-matching samplings $(O_n), (P_n)$ respectively, then there exists an equality-matching sampling $(L_n)$ for $T \ceq T_1 \cup T_2$ with $\lvert L_n \rvert = \lvert O_n \rvert \cdot \lvert P_n \rvert$ for all $n\in \N$. If $(O_n)$ and $(P_n)$ are computable in polyonmial time, then $(L_n)$ is also computable in polynomial time.
\end{theorem}

\begin{proof}
Let $(O_n)$ and $(P_n)$ be  equality-matching samplings of $T_1, T_2$, respectively.
Fix $n\in \N$ and let $\set{\fB_{i,1}, \dots, \fB_{i,p_i}}$ be $O_n$ for $i=1$ and $P_n$ for $i=2$.
% Fix $n$ and let $L_{i,n} = \set{\fB_{i,1}, \dots, \fB_{i,p_i}}$ for $i\in [2]$ be the $n$-samples of equality-matching sampling of $T_1, T_2$, respectively. 
For all $i\in [p_1], j\in [p_2]$ we now define a $\tau_1 \cup \tau_2$-structure $\fM_{i,j}$ with domain $M_{i,j}\ceq B_{1,i} \times B_{2,j}$ as follows.

For $R \in \tau_1$ of arity $k$ we define
\begin{align*}
&((a_1, b_1), \dots, (a_k, b_k)) \in R^{\fM_{i,j}} \Leftrightarrow \\
&(\forall i,j. (a_i = a_j \Leftrightarrow b_i = b_j)) \AND (a_1, \dots, a_k) \in R^{\fB_{1,i}}.
\end{align*}
Analogously, we define for $R \in \tau_2$ of arity $k$ 
\begin{align*}
&((a_1, b_1), \dots, (a_k, b_k)) \in R^{\fM_{i,j}} \Leftrightarrow \\
&(\forall i,j. (a_i = a_j \Leftrightarrow b_i = b_j)) \AND (b_1, \dots, b_k) \in R^{\fB_{2,j}}.
\end{align*}

Let $\phi_i(x_1, \dots, x_n)$ be a conjunction of atomic $\tau_i$-formulas for $i=1$ and $i=2$. Notice that forcing $\phi_1$ and $\phi_2$ to have the same variables can be done without loss of generality by introduction of dummy constraints like $x=x$.
Suppose there exists a homomorphism $g \colon D(\phi_1(x_1, \dots, x_n) \AND \phi_2(x_1, \dots, x_n)) \rightarrow \fA$ where $\fA$ is a model of $T$. Then $g\colon D(\phi_i(x_1, \dots, x_n)) \rightarrow \fA^{\tau_i}$ is also a homomorphism for $i=1$ and $i=2$. As the chosen samplings are equality-matching, there exists $u_i \in [p_i]$ and homomorphisms $h_i \colon D(\phi_i(x_1, \dots, x_n)) \rightarrow \fB_{i,u_i}$ for $i=1$ and $i=2$ such that $h_1$ and $h_2$ both identify variables in the same way that $g$ does. Then, by construction, $x_i \mapsto (h_1(x_i), h_2(x_i))$ defines a homomorphism from $D(\phi_1 \AND \phi_2)$ to $\fM_{u_1, u_2}$ identifying the same variables as $g$.

For the reverse direction suppose that $h$ is a homomomorphism from %$D(\phi_1(x_1,\dots, x_n) \AND \phi_2(x_1, \dots, x_2))$
$D(\phi_1 \AND \phi_2)$ to $\fM_{u_1, u_2}$.
We want to show the existence of a homomorphism $g$ from 
%$D(\phi_1(x_1,\dots, x_n) \AND \phi_2(x_1, \dots, x_2))$ 
$D(\phi_1 \AND \phi_2)$ to $\fA$, where $\fA$ is a model of $T$ and $g$ and $h$ identify the same variables. 
To simplify presentation, we would like to assume that $h$ is injective.
Suppose there exists $x_i, x_j$ such that $h(x_i) = h(x_j)$. Then we replace all occurrences of $x_j$ in $\phi_1$ and $\phi_2$ by $x_i$ and iterate until $h$ is injective on the remaining variables. Call the resulting formulas $\phi_1'$ and $\phi_2'$. If there exists an injective homomorphism $g$ from $D(\phi_1' \AND \phi_2')$ to a model of $T$, we can extend $g$ to the formerly substituted variables via $g(x_j) \ceq g(x_i)$ to get a homomorphism from $D(\phi_1 \AND \phi_2)$ to some model of $T$ which has the same identifications as $h$. Hence, it is sufficient to prove the existence of $g$ under the assumption that $h$ is injective.

Let $h_i \ceq \pi_i^2 \circ h$ for $i\in [2]$, where $\pi^2_i$ is the projection to the $i$-th coordinate out of of two coordinates. Now suppose there exists a set $S \subseteq h(D(\phi))$ of size at least two such that any two elements in $S$ are equal in the first coordinate (the case with the second coordinate can be proven analogously). Without loss of generality we will assume $S=\set{h(x_1), \dots, h(x_k)}$. Furthermore, we may assume that $S$ is maximal, i.e., for all $i>k$ we have $h_1(x_i) \neq h_1(x_1)$.
For $i\in [k]$ let $\psi_i(x_i, x_{k+1}, \dots, x_n)$ be the conjunction of all conjuncts in $\phi_1(x_1, \dots, x_n)$ which do not contain $x_j$ for all $j\in [k]\setminus\set{i}$. 
Then $\phi_1$ is equivalent to $\bigwedge_{i=1}^k \psi_i(x_i, x_{k+1}, \dots, x_n)$ because no conjunct of $\phi_1$ can contain more than one variable from $x_1, \dots, x_k$, as this conjunct would not be satisfied in the image of $h_1$ in $\fM_{u_1, u_2}$. 
As $T_1$ has an equality-matching sampling, there exists a homomorphism $g_1 \colon D(\phi_1) \rightarrow \fA_1$ for some model $\fA_1$ of $T_1$, such that $g_1$ has the same identifications as $h_1$. 
By the definition of no pp-algebraicity, the set 
 \begin{displaymath}
 \set{a\in A \mid \fA_1 \models \psi_i(a, g_1(x_{k+1}), \dots, g_1(x_n)) }
 \end{displaymath}
 is infinite for all $i\in [k]$. 
Hence, there exists a homomorphism $g_1' \colon D(\phi_1) \rightarrow \fA_1$ such that $g_1'(x_i) \neq g_1'(x_j) \neq g_1'(x_l) = g_1(x_l)$ for all $i,j \in [k]$ and $l\in [n]\setminus [k]$. 
As the sampling for $T_1$ is equality-matching, there exists $u_1'\in [p_1]$ and a homomorphism $h_1' \colon D(\phi_1) \rightarrow \fB_{1,u_1'}$ with the same identifications as $g_1'$. This yields a homomorphism $x_i \mapsto (h_1'(x_i), h_2(x_i))$ to $\fM_{u_1', u_2}$ where strictly less variables are mapped to the same row or column than by $h$. 

Iterated application of this argument to $h_1$ and, in an analogous way, to $h_2$ proves that there exist injective homomorphisms $h_1^*$ and $h_2^*$ such that $h^*$ defined by  $x_i \mapsto (h_1^*(x_i), h_2^*(x_i))$ is a homomorphism from $D(\phi_1(x_1,\dots, x_n) \AND \phi_2(x_1, \dots, x_2))$ to $\fM_{u_1^*, u_2^*}$. 
As $(O_n)$ and $(P_n)$ are equality-matching samplings there exist injective homomorphisms $g_i^* \colon D(\phi_i(x_1, \dots, x_n)) \rightarrow \fA_i^*$ where $\fA_i^*$ is a model of $T_i$ for $i=1$ and $i=2$. Now observe that $\fA_1^*$ and $\fA_2^*$ are both either empty (in which case the theorem becomes trivial) or infinite because no pp-algebraicity implies that $\set{x \mid \fA_i^* \models x=x}$ is infinite or empty. If $A_i^*$ is uncountable, then there exists a $\fA_i'$ with countable domain such that $\Th(\fA_i') = \Th(\fA_i^*)$ due to the Downward L\"owenheim-Skolem Theorem (see~\cite{HodgesLong}, page~90) and we may substitute $\fA_i^*$ by $\fA_i'$. Hence, we may assume that $\lvert A_1^* \rvert = \lvert A_2^* \rvert$ and therefore, there exists a bijection $f\colon A_1^* \rightarrow A_2^*$ such that $f(g_1^*(x_i)) = g_2^*(x_i)$ for all $i\in [n]$. However, any bijection from $A_1^*$ to $A_2^*$ induces a fusion $\fA$ of $\fA_1^*$ and $\fA_2^*$ with domain $A_1^*$ and therefore a model of $T$ by Proposition~\ref{thm:fusionIsModel}. Clearly, $g \colon D(\phi) \rightarrow \fA^*$ defined by $g(x_i) \ceq g_1^*(x_i)$ for all $i\in [n]$ is an injective homomorphism. Hence, $(L_n)$ with $L_n \ceq \set{\fM_{1,1}, \dots, \fM_{p_1, p_2}}$ is an equality-matching sampling for $T$.

We have $\vert L_n \rvert =\sum_{i\in [p_1], j\in [p_2]} \lvert \fB_{1,i} \rvert \cdot \lvert \fB_{2,j} \rvert =\lvert O_n \rvert \cdot \lvert P_n \rvert $ for all $n\in \N$ and to determine whether a relation $R^{\fM_{i,j}}$ holds on an tuple $t$ requires at most $n^2$ checks for equality and a check if $t\in R^{\fB_{1,i}}$ or $t\in R^{\fB_{2,j}}$. Therefore, if both samplings are computable in polynomial time, $(L_n)$ is computable in polynomial time as well.
\end{proof}

\begin{remark}
 Notice that with almost the same construction, an equality-matching sampling for any finite number of theories can be constructed as long as all theories have an equality-matching sampling and only models without pp-algebraicity. For that we use a higher-dimensional version of the matrix defined in the proof, where a relation $R\in \tau_i$ holds on a tuple $t$ if only if any two entries of $t$ either differ in all their coordinates or are equal in all their coordinates and the projection to the $i$-th dimension of each element of $t$ is in $R$ as defined in the respective sample for $T_i$.
\end{remark}

Note that if $\fA$ has a finite signature and is without pp-algebraicity, then all models of $\Th(\fA)$ are  without pp-algebraicity. This holds because for each primitive positive formula $\phi(x_1, \dots, x_k)$ and $n\in \N$, the theory $\Th(\fA)$ contains a sentence expressing that for any choice of parameters $x_2, \dots, x_k$ any satisfying assignment for $x_1$ is either equal to one of the parameters or there are at least $n$ different satisfying assignments for $x_1$.

If one of the theories, say $T_1$, does not satisfy the conditions of Theorem~\ref{thm:Both_NoAlg+InjSamp} because it has models with pp-algebraicity, we can prove the following.

\begin{proposition}\label{thm:InjSamp+ECL}
 Let $T_1$ be a theory with finite relational signature and let $\fA_2$ be an infinite structure with signature $\tau_2$ whose relations are first-order definable over the empty signature. If $T_1$ has an equality-matching sampling $(L_n)$ and only infinite models, then the expansion of the samples in $(L_n)$ with relations in $\tau_2$ by their first-order definitions is an equality-matching sampling for $T_1 \cup \Th(\fA_2)$.
\end{proposition}

\begin{proof}
Fix $n \in \N$ and let $\tau_1$ be the signature of $T_1$. Call the expanded sampling $(L_n')$.
Let $\phi_i$ be a conjunction of atomic $\tau_i$ formulas.
Then $\phi_2(x_1, \dots, x_n)$ is equivalent to $\bigvee_{i\in [k]} \psi_i(x_1, \dots, x_n)$, where each $\psi_i$ specifies by a conjunction of equalities and disequalities which variables must be equal and which must differ . However, for every $i\in [k]$ the formula $\phi_1 \AND \psi_i$ is satisfiable in some element of $L_n$ if and only if it is satisfiable in some model of $T_1$ because $(L_n)$ is equality-matching. Therefore, $\phi_1 \AND \phi_2$ is satisfiable in some model of $T_1$ if and only if it is satisfiable in some structure in $L_n'$.
\end{proof}

There are also theories with a 
%decidable CSP 
sampling 
but without an equality-matching sampling: 
for example, $\fA \ceq (\Z; R_a,R_m)$ where
$R_a := \{(x,y,z) \mid x+y=z\}$ and $R_m := \{(x,y,z) \mid x\cdot y = z\}$ has a trivial CSP and therefore a sampling. If there exists an equality-matching sampling for $\Th(\fA)$ we can construct a sampling for $T \ceq \Th(\fA) \cup \Th(\Z; \neq)$ by  Proposition~\ref{thm:InjSamp+ECL}. Then we can define $z=1$ by the primitive positive formula $\exists x,y \, (x\cdot z = x \AND y\cdot z = y \AND x\neq y)$. 
%and therefore, primitive positive formulas can encode diophantine equations in $T$. 
It is now easy to see that Hilbert's tenth problem, which is undecidable, is many-to-one reducible to $\CSP(T)$, contradicting Proposition~\ref{thm:decidableEqSampling}.

The following example shows that we cannot drop the assumption of no pp-algebraicity in Theorem~\ref{thm:Both_NoAlg+InjSamp}.

\begin{example}\label{ex:Succ2Col}
 Let $\succ$ be the successor relation on $\N$. Let $P_0, P_1$ be two disjoint, infinite sets (two colours). Then $T_1 \ceq \Th(\N; \succ)$ and $T_2 \ceq \Th(P_0 \cupdot P_1; P_0, P_1)$ both have an equality-matching sampling which can be computed in polynomial time. However, any sampling $(L_n)_{n\in \N}$ of $T_1 \cup T_2$ has $\lvert L_n \rvert \geq 2^n$ for all $n\in \N$, and therefore no polynomial sampling exists.

 To prove this, first notice that we can construct a model $\fA$ for $T \ceq T_1 \cup T_2$ such that $\CSP(\fA) = \CSP(T)$. This can be done by giving each natural number the colour $P_0$ or $P_1$ in such a way that the resulting sequence over $\set{0,1}$ includes any binary number as consecutive subsequence. 
  Therefore, when examining lower bounds for $\lvert L_n \rvert$, we may assume by Proposition~\ref{thm:samplesOfStructures} that for each $n$ there is only one element  $\fB_n \in S_n(T)$ in $L_n$. Now observe the following:
 \begin{enumerate}
  \item In any $n$-sample $\fB_n$, no element $x$ can satisfy $P_0(x) \AND P_1(x)$ and there is no pp-formula which can force a node to have no colour. Hence, if some node in $\fB_n$ has no colour, we can colour this node arbitrarily and the result is again in $S_n(T)$. Therefore, without loss of generality, each node in the sampling has exactly one colour.
  \item When we fix a node $c$ in $\fB_n$ such that  $\phi_{c} \ceq \succ(c, x_2) \AND \dots \AND \succ(x_{n-1}, x_n)$ is satisfiable in $\fB_n$, then there exists exactly one map $s\colon [n] \rightarrow \set{0,1}$ such that $\psi_{c,s} \ceq \phi_{c} \AND \bigwedge_{i\in [n]} P_{s(i)}(x_{i})$ is satisfiable in $\fB_n$. If there were two distinct maps $s$ and $s'$, then $\psi_{x_1,s} \AND \psi_{x_1,s'}$ would be satisfiable in $\fB_n$, but not in $\fA$.\label{item:uniqueSequence}
 \end{enumerate}
Now observe that atomic $\tau_1 \cup \tau_2$-formulas with $n$ variables can encode a binary number $b$ of length $n$ via $\succ(x_1, x_2) \AND \dots \AND \succ(x_{n-1}, x_n) \AND \bigwedge_{i\in [n]} P_{b_i}(x_i)$. By item~\eqref{item:uniqueSequence} each element of $\fB_n$ is the start of at most one binary number and therefore, $\fB_n$ must have at least $2^n$ elements.

We would like to remark that for all $n$ there exists $\fB_n \in S_n(T)$ of size $2^n$ such that $(\fB_n)$ is an equality-matching sampling of $T$. It can be constructed via a de-Brujin sequence. However, the smallest $n$-universal substructure of a model of $T$ has size $2^n+n-1$, which is larger.
\end{example}

Even though we cannot drop no-pp-algebraicity in general, there are structures with pp-algebraicity where the sets definable by primitive positive formulas are somewhat tame and an equality-matching sampling of the union of their theories exists and is computable in polynomial time. We will demonstrate this with the following example, which will also occur in the next section.

\begin{example}\label{ex:EEMaj}
 Let $E_1$ be the relation $$\set{(2a,2a+1), (2a+1, 2a) \mid a\in \N}.$$ Let $E_2$ be a copy of $E_1$. Then we consider $T \ceq \Th(\N; E_1) \cup \Th(\N; E_2)$. For $n\in \N$ we will now construct $\fB_n \in S_n(T)$ such that $(\fB_n)_{n\in \N}$ is an equality-matching sampling for $T$ computable in polynomial time.

 Consider formulas $\gamma_k$ and $\delta_k$ describing an \emph{alternating sequence} and an \emph{alternating cycle}, both of length $2k$, respectively.
 Formally, we define $ \gamma_k(x_1, \dots, x_{2k+1})$ as the formula $\bigwedge_{i \in [k]} (E_1(x_{2i-1}, x_{2i}) \AND E_2(x_{2i}, x_{2i+1}))$ and $\delta_k(x_1, \dots, x_{2k})$ as the formula $\gamma_k(x_1, \dots, x_{2k}, x_1)$.

 Let $\fB_n$ be the structure containing $\lceil\frac{n}{2k}\rceil$ many copies of $D(\delta_k(x_1,x_2, \dots, x_{2k}))$, for each $k\in \set{1, \dots, \lceil \frac{n}{2} \rceil}$. Such $\fB_n$ has size in $O(n^2)$.

 We now prove that $(\fB_n)_{n\in \N}$ is a sampling for $T$. To this end, note that $\fB_n$ can be extended to a model of $T$ by adding infinitely many alternating cycles of all positive even lengths. 
 We will now analyze which instances $\phi$ of $\CSP(T)$ are satisfiable and prove that these are satisfiable in $\fB_n$.
 Let $\phi_i$ be a conjunction of atomic $\tau_i$-formulas for $i=1$ and $i=2$ and suppose that $\phi\ceq \phi_1 \AND \phi_2$ contains only variables from $\set{x_1, \dots, x_n}$ and is satisfiable in some model of $T$. 
 If $\phi$ contains conjuncts of the form $E_i(x,y) \AND E_i(y,z)$ then $x=z$ must hold in any model of $T$. Hence, we may simplify our analysis by replacing $z$ by $x$ in $\phi$. The instance $E_i(x,x)$ is never satisfiable. Therefore, if $\phi$ contains the conjunct $E_i(x,y)$ then $x$ and $y$ do not appear in any other $E_i$ conjunct, i.e., each variable appears at most once in all $E_1$ conjuncts and at most once in all $E_2$ conjuncts. Hence, $\phi$ can only describe alternating sequences and alternating cycles, both of which homomorphically map into alternating cycles of suitable length and therefore to $\fB_n$.
\end{example}

\section{Exemplary Application to CSPs}\label{sec:application}
We finally show how to use the sampling $(L_n)$ constructed in Theorem~\ref{thm:Both_NoAlg+InjSamp} in order to prove the polynomial-time tractability of $\CSP(T_1 \cup T_2)$. 
To do this, we use algorithms for the CSPs of the structures $\fB_n \in L_n$ which run in \emph{uniform} polynomial time, i.e., the runtime is polynomial in the size of the instance and in $\lvert B_n \rvert$.
The following generalises a result for CSPs of structures by Bodirsky, Macpherson, and Thapper
~\cite{BodirskyMacphersonThapper13}
to CSPs of theories. 

\begin{proposition}\label{prop:sampling}
  Let $T$ be a theory with finite relational signature and let $(L_n)$ be a sampling of $T$ computable in polynomial time. If there is an algorithm that solves $\Csp(\fB_n)$ for every $\fB_n \in L_n$ in uniform polynomial time, then $\Csp(T)$ is polynomial-time tractable. 
\end{proposition}
\begin{proof}
 To decide an instance $\phi$ of $\CSP(T)$  on $n$ variables we run the algorithm on $(\fB, \phi)$ for all $\fB  \in L_n$. If the algorithm accepts on any of these tuples, we return \enquote{satisfiable}, otherwise we return \enquote{unsatisfiable}. This is correct and complete by the definition of sampling. Furthermore, the procedure runs in polynomial time, because $|L_n|$ is polynomial in $n$ and the algorithms worst-case runtime is polynomial in $n$ and $|L_n|$ and therefore polynomial in $n$.
 \end{proof}

To describe classes of finite structures whose CSP can be be solved by a uniform polynomial-time algorithm, the following concepts from universal algebra are important. 

\begin{definition}
 Let $\fA$ be a relational structure with signature $\tau$ and let $f$ be an operation of arity $k$ on $\fA$. 
 We call $f$ a \emph{polymorphism} of $\fA$ if it is a homomorphism from $\fA^k$ to $\fA$, i.e., whenever $t_1, \dots, t_k \in R^{\fA}$, for some $R\in \tau$, then $f(t_1, \dots, t_k) \in R^{\fA}$, where $f$ is applied componentwise.
   An operation $f \colon A^k \rightarrow A$ is \emph{totally symmetric} if for all $x_1, \dotsc, x_k, y_1, \dotsc, y_k\in A$ we have
     $f(x_1,\dotsc, x_k) = f(y_1, \dotsc, y_k)$ whenever $\set{x_1, \dotsc, x_k} = \set{y_1, \dotsc, y_k}$.
      If $k \geq 3$ and for all $a,b\in A$ we have $f(a,b,b,\dots, b) = f(b,a,b,b,\dots, b) = \dots = f(b,b,\dots, b,a) = b$, then $f$ is called a \emph{near-unanimity operation}. 
\end{definition}

\begin{proposition}[Corollary~3.6 in~\cite{JeavonsCohenCooper98}]\label{thm:nearUnanimity_tractable}
There is an algorithm that solves $\CSP(\fB)$
in uniform polynomial time for all finite relational structures $\fB$ with a near-unanimity polymorphism.
\end{proposition}

Results about the applicability of fast uniform algorithms for CSPs of finite relational structures with a near-unanimity polymorphism can be found in the work of Kozik (2016).

Proposition~\ref{thm:nearUnanimity_tractable} and Proposition~\ref{prop:sampling} can be applied to Example~\ref{ex:EEMaj} with the following near-unanimity operation $f$ defined on $\fB_n$ for $n\in \N$:
 \begin{displaymath}
    f(x,y,z) \ceq \left\{ \begin{array}{ll}
                     y         & \text{if } y = z,\\
                     x     & \text{otherwise.}
                     \end{array} \right.
  \end{displaymath}
Note that $f$ is a polymorphism of $\fB_n$, for all $n$, because if $(t_{1,1}, t_{1,2})$, $(t_{2,1}, t_{2,2})$ and $(t_{3,1}, t_{3,2})$ are in $E_k$ and $t_{i,1} = t_{j,1}$, then $t_{i,2} = t_{j,2}$ and therefore $f(t_1, t_2, t_3)= (t_{i,1}, t_{i,2}) \in E_k$ for all $k\in [2], i,j \in [3], i\neq j$. 
In particular, $\CSP(T)$ with $T$ from Example~\ref{ex:EEMaj} can be solved in polynomial time.

 An example for a totally symmetric polymorphism is the minimum operation over $(\Q;<)$. Also, we can define minimum as an operation of any arity. 
The following result builds on results of Feder and Vardi~\cite{FederVardi}.

\begin{theorem}[Section~3 in~\cite{DalmauPearson}]
\label{thm:finiteStructureSolvedByAC}
 Let $\fB$ be a finite structure with finite relational signature. Then the arc-consistency procedure solves $\CSP(\fB)$ iff $\fB$ has totally symmetric polymorphisms of all arities.
\end{theorem}

The Arc-consistency algorithm runs in worst-case time $O(n^2m^2)$ where $n$ is the number of variables in the instance and $m$ is the size of the domain~\cite{Cooper89_AnOptimalKConsistencyAlgorithm}. 
Theorem~\ref{thm:finiteStructureSolvedByAC} 
and Proposition~\ref{prop:sampling} immediately yield the following. 

\begin{corollary}\label{thm:samplesWithTS}
 Let $(L_n)$ be a sampling for $T$ which is computable in polynomial time. If all samples in $(L_n)$ have totally symmetric polymorphism of all arities, then $\CSP(T)$ is polynomial-time tractable.      
\end{corollary}

We can slightly relax the requirement that all the structures in $L_n$ have totally symmetric polymorphisms.

\begin{lemma}\label{thm:sampleHasHomToTotallySym}
Let $T$ be a theory with finite relational signature and let $(L_n)$ be a sampling of $T$ computable in polynomial time. If for all $n\in \N$ and all $\fB \in L_n$ there exists a model $\fA_{\fB}$ of $T$ and a homomorphism $h\colon \fB \rightarrow \fA_{\fB}$ such that $h(\fB)$ has totally symmetric polymorphisms of all arities, then $\CSP(T)$ is polynomial-time tractable.
\end{lemma}

\begin{proof}
To decide an instance with $n$ variables, we run arc-consistency on each pair $(\fB, \phi)$ where $\phi$ is the instance and $\fB \in L_n$. Whenever arc-consistency returns \enquote{unsatisfiable} for an instance on a structure $\fB$, the instance is unsatisfiable in $\fB$ (this holds for any finite structure). Hence, if all runs reject, the instance is unsatisfiable.

 Suppose that the arc-consistency procedure returned \enquote{satisfiable} on $(\fB, \phi)$. Then, the algorithm produced for each variable a non-empty list of values in $\fB$ such that no list can be shortened due to constraints in the instance. Thus, when we apply $h$ to all these values, we gain non-empty lists of values over $\fA_{\fB}$ where all relations that did hold in $\fB$ hold again. Hence, arc-consistency cannot return \enquote{unsatisfiable} when run on the image of $h$. 
 By Theorem~4.4, this means that the instance is satisfiable on the image of $h$, therefore in $\fA_{\fB}$ and therefore in some element of $(L_n)$, by definition of sampling. 
 As $(L_n)$ is computable in polynomial time and arc-consistency runs in uniform polynomial time, we can decide $\CSP(T)$ in polynomial time.
\end{proof}

We conclude this section with a concrete reasoning task that can be
solved in polynomial time with our methods,
and that is not covered by the combination results of Nelson and Oppen~\cite{NelsonOppen79, Oppen80_ComplexityConvexityAndCombinationsOfTheories} . 

\begin{example}\label{ex:rminPQ}
Given parts of a huge machine, each of which can be mounted by exactly one of two robots, and precedence constraints on the order in which parts can be mounted, calculate a possible order in which the machine can be assembled. Some parts must be mounted by one robot, some by the other robot, and some parts can be mounted by both robots. The two robots are not allowed to work simultaneously. This reasoning task can be modelled as $\CSP(T_1 \cup T_2)$ where 
\begin{itemize}
\item $T_1 := \Th(\fA_1)$ for a first-order expansion $\fA_1$ of $(\Q;<)$ 
that allows to model the precedence constraints, 
and 
\item $T_2 := \Th(\fA_2)$ where $\fA_2$ has two disjoint countably infinite unary relations and their union as domain. 
\end{itemize}
By Lemma~\ref{thm:QhasSampling} and
Lemma~\ref{thm:baseStrHaveSampling},
the theories $T_1$ and $T_2$ have polynomial-time computable samplings, and both structures have no pp-algebraicity. 
Hence, by Theorem~\ref{thm:Both_NoAlg+InjSamp}
 there exists a polynomial-time computable sampling $(L_n)$ for $T_1 \cup T_2$. 
 Let $\tau_i$ be the signature of $\fA_i$, for 
 $i \in \{1,2\}$. 
 For every $n \in {\mathbb N}$ and every $\fC_n \in L_n$ there exists an injective homomorphism $h_1$  from 
 $\fC_n^{\tau_1}$ to $\fA_1$ and an injective homomorphism $h_2$ from 
 $\fC_n^{\tau_2}$ to $\fA_2$; this follows by inspection of the construction of $\fC_n$ in the proof of Theorem~\ref{thm:Both_NoAlg+InjSamp}
 and the proofs of Lemma~\ref{thm:QhasSampling} and
Lemma~\ref{thm:baseStrHaveSampling}. 
As $\fA_1$ and $\fA_2$ are both countably infinite and $h_1$ and $h_2$ are injective, there exists a bijection $f$ between $A_1$ and $A_2$ such that $f(h_1(a)) = h_2(a)$ for all $a \in C_n$, inducing a fusion $\fA$ of $\fA_1$ and $\fA_2$. 
By construction of $\fA$, the map $h_1$ is a homomorphism from $\fC_n$ to $\fA$. 

Also note that the minimum operation is a polymorphism of all unary relations. Hence, 
 if $\fA_1$ has the minimum operation
as a polymorphism then the minimum operation is also a polymorphism of $\fA$. 
Then Lemma~\ref{thm:sampleHasHomToTotallySym}
 implies
 that $\Csp(T_1 \cup T_2)$ is polynomial-time tractable. 

An example for such a structure $\fA_1$ is $\fA_1 \ceq (\Q; <, R)$ where $R$ is the ternary relation defined by $R(x,y,z) \Leftrightarrow x=\min(y,z)$;
it is easy to check that the minimum operation is a polymorphism of $\fA_1$. 
This example is not covered by the conditions of Nelson and Oppen, since 
$\Th(\fA_1)$ is \emph{not convex}:
the formula $R(x,y,z) \wedge x \neq y$ is satisfiable,
the formula $R(x,y,z) \wedge y \neq z$ is satisfiable, but the formula $R(x,y,z) \wedge x \neq y \wedge y \neq z$ is not satisfiable in $\fA_1$. 
\end{example}

We mention that the polynomial-time tractability of 
$\Csp(\fA_1)$ for $\fA_1 = ({\mathbb Q};<,R)$ in our example was already known in the SMT community~\cite{Max-atoms}. 
It is straightforward to generalise the example to any finite number $m$ of robots, by replacing the structure $\fA_2$ by 
the structure $({\mathbb Q};P_1,\dots,P_m)$ that we have already encountered in Lemma~\ref{thm:baseStrHaveSampling}.

\printbibliography
%\bibliographystyle{numeric}
%\bibliography{local}

\end{document}